\documentclass[12pt,a4paper]{article}
\usepackage{graphicx,fullpage}
\newcommand{\coverpage}[3]{\thispagestyle{empty}
    \addtocounter{page}{-1}
\null\vspace*{-1cm} \hfill\includegraphics[scale=0.2]{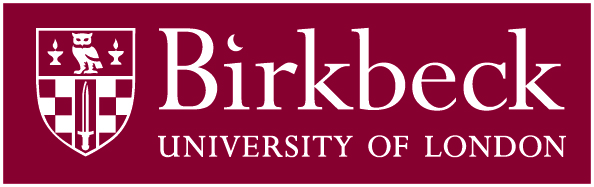} \vskip 2in
\begin{center} \begin{minipage}{0.7\textwidth}\begin{center}\Huge\bf{#1}\end{center} \end{minipage}\end{center}  \vfill
\begin{center} {\large By}\bigskip\\ {\large #2}\\ \end{center} \vfill 
 \framebox{\begin{minipage}{\textwidth}
Birkbeck Mathematics Preprint Series\hfill
Preprint Number #3 \\ \\
\null\hfill www.bbk.ac.uk/ems/research/pure/preprints 
\end{minipage}}
\newpage}

\usepackage{amsmath,amsthm,amssymb}
\usepackage{times}
\usepackage{enumerate}
\usepackage{mathtools}
\usepackage[section]{algorithm}
\usepackage{algpseudocode}
\usepackage[section]{placeins}

\algnewcommand\True{\textbf{true}\space}
\algnewcommand\False{\textbf{false}\space}
\algnewcommand\Not{\textbf{not}\space}
\algnewcommand\algorithmicforeach{\textbf{for each}}
\algdef{S}[FOR]{ForEach}[1]{\algorithmicforeach\ #1\ \algorithmicdo}

\newtheorem{thm}{Theorem}[section]

\newtheorem*{theorem-non}{Theorem}

\newtheorem{lemma}[thm]{Lemma}
\newtheorem{prop}[thm]{Proposition}
\newtheorem{cor}[thm]{Corollary}
\newtheorem{conj}[thm]{Conjecture}
\theoremstyle{plain}
\newtheorem{obs}[thm]{Observation}
\theoremstyle{definition}
\newtheorem{notation}[thm]{Notation}

\begin{document}

\coverpage{Groups whose locally maximal product-free sets are complete}{Chimere S. Anabanti, Grahame Erskine and Sarah B. Hart}{25}

\title{Groups whose locally maximal product-free sets are complete}
\author{Chimere S. Anabanti \thanks{The first author is supported by a Birkbeck PhD Scholarship}\\ c.anabanti@mail.bbk.ac.uk \and Grahame Erskine \\ grahame.erskine@open.ac.uk\and Sarah B. Hart\\ s.hart@bbk.ac.uk}
\date{}

\maketitle

\begin{abstract}
\noindent Let $G$ be a finite group and $S$ a subset of $G$. Then $S$ is {\em product-free} if $S \cap SS = \emptyset$, and {\em complete}  if $G^{\ast} \subseteq S \cup SS$.  A product-free set is {\em locally maximal} if it is not contained in a strictly larger product-free set. If $S$ is product-free and complete then $S$ is locally maximal, but the converse does not necessarily hold. Street and Whitehead \cite{SW1974} 
 defined a group $G$ as {\em filled} if every locally maximal product-free set $S$ in $G$ is complete (the term comes from their use of the phrase `{\em $S$ fills $G$}' to mean $S$ is complete). They classified all abelian filled groups, and conjectured that the finite dihedral group of order $2n$ is not filled when $n=6k+1$ ($k\geq 1$). The conjecture was disproved by two of the current authors in \cite{AH2015}
 , where we also classified the filled groups of odd order.  In this paper we classify filled dihedral groups, filled nilpotent groups and filled groups of order $2^np$ where $p$ is an odd prime. We use these results to determine all filled groups of order up to 2000. 
\end{abstract}

\section{Preliminaries}
Let $S$ be a non-empty subset of a group $G$. We say $S$ is {\em product-free} if $S\cap SS=\varnothing$, where $SS = \{ab: a, b \in S\}$ (in particular $a$ and $b$ are not necessarily distinct).  A product-free set $S$ is said to be {\em locally maximal} if whenever $\Sigma$ is product-free in $G$ and $S\subseteq \Sigma$, then $S=\Sigma$. A product-free set $S$ of $G$ is {\em complete} (or, equivalently, {\em fills} $G$) if $G^*\subseteq  S \cup SS$ (where $G^*$ is the set of all non-identity elements of $G$).  Complete product-free sets have cropped up in many areas; some of these were described by Cameron in \cite{cameron1}; see also \cite{cameron2}. For example, complete product-free symmetric subsets of groups give rise in a natural way to regular triangle-free graphs of diameter 2. In an investigation of these graphs, Hanson and Seyffarth \cite{hs} showed, by exhibiting small complete product-free sets of cyclic groups, that the smallest valency of such a graph on $n$ vertices is within a constant factor of the trivial bound of $\sqrt{n}$. In the abelian case the term {\em sum-free} is used. Complete sum-free sets have also been investigated by Payne \cite{payne}, Calkin and Cameron \cite{calkin}, amongst others. There is also a link to finite geometry:  complete sum-free sets in the elementary abelian group of order $2^{n}$ correspond to complete caps in the projective space $PG(n-1,2)$; that is, collections of points,  maximal by inclusion, with no three collinear. However there does not seem to have been much work done for non-abelian groups, where of course we speak of product-free sets. It's easy to see that all complete product-free sets are locally maximal. The converse often holds in that many locally maximal product-free sets are complete. In the extreme case we say $G$ is a {\em filled group} if every locally maximal product-free set in $G$ is complete (or, using the notation of Street and Whitehead \cite{SW1974}, every locally maximal product-free set in $G$ fills $G$).  The definition of a filled group, due to Street and Whitehead, was motivated by the observation that a product-free set in an elementary abelian 2-group $A$ is locally maximal if and only if it fills $A$, and hence the elementary abelian 2-groups are filled groups. They asked which other groups, if any, are filled. In \cite{SW1974} they classified the filled abelian groups and the first few dihedral groups. In this paper, we classify filled groups of various kinds. In Section \ref{secdih} we deal with dihedral groups. Section \ref{secnil} covers nilpotent groups. Section \ref{sec2np} looks at groups of order $2^np$ where $p$ is an odd prime and $n$ is a positive integer. Finally in Section \ref{secalg} we describe algorithms which we have implemented in GAP \cite{GAP4}, that allow us to check for filled non-nilpotent groups of all orders up to 2000. In the rest of this section we establish notation and state some known results.

Throughout this paper, we write $C_n$ for the cyclic group of order $n$ and $D_{2n}$ for the dihedral group of order $2n$.  All groups in this paper are finite.

\begin{notation}\label{ts}
Let $S$ be a subset of a group $G$. We define $S^{-1}$, $T(S)$ and $\sqrt S$ as follows:
\begin{align*} 
S^{-1} &=\{s^{-1} : s \in S\};\\
T(S) &=S \cup SS \cup SS^{-1} \cup S^{-1}S;\\
\sqrt{S} &= \{x \in G:x^2 \in S\}.\end{align*}
\end{notation}

We end this section with the following result, which gathers together some useful facts that we will need.

\begin{thm}\label{T1}\begin{enumerate} 
\item[(i)] \cite[Lemma 3.1]{GH2009} Let $S$ be a product-free set in a finite group $G$. Then $S$ is locally maximal if and only if $G=T(S) \cup \sqrt{S}$.
\item [(ii)] \cite[Lemma 1]{SW1974} If $G$ is a filled group, and $N$ is a normal subgroup of $G$, then $G/N$ is filled.
\item [(iii)]\cite[Theorem 2]{SW1974} A finite abelian group is filled if and only if it is $C_3$, $C_5$ or an elementary abelian 2-group.
\item [(iv)] \cite[Lemma 2.3]{AH2015} The only filled group with a normal subgroup of index 3 is $C_3$.
\item [(v)] \cite[Lemma 2.5]{AH2015} If $G$ is a filled group with a normal subgroup $N$ of index $5$ such that not every element of order $5$ is contained in $N$, then $G \cong C_5$.
\item [(vi)]\cite[Theorem 2.6]{AH2015} The only filled groups of odd order are $C_3$ and $C_5$.
\item [(vii)] \cite[Prop 2.8]{AH2015} For $n \geq 2$, the generalized quaternion group of order $4n$ is not filled.
\end{enumerate}\end{thm}

\section{Dihedral groups}\label{secdih}
A list of non-abelian filled groups of order less than or equal to $32$ was given in \cite{AH2015}. There are eight such groups: six are dihedral, and the remaining two are non-dihedral 2-groups. The dihedral groups on the list are those of order 6, 8, 10, 12, 14 and 22. Our aim in this section is to show that these are in fact the only filled dihedral groups. 

\begin{notation}
We write $D_{2n}=\langle x,y|~x^{n}=y^2=1, xy=yx^{-1} \rangle$ for the dihedral group of order $2n$ (where $n>2$). In $D_{2n}$, the elements of $\langle x\rangle$ are called rotations and the elements of $\langle x\rangle y$ are called reflections. For any subset $S$ of $D_{2n}$, we write $A(S)$ for $S \cap \langle x \rangle$, the set of rotations of $S$, and $B(S)$ for $S \cap \langle x \rangle y$, the set of reflections of $S$.
\end{notation}

\begin{obs}\label{obs1}
Suppose $S$ is a subset of $D_{2n}$. Let $A = A(S)$ and $B = B(S)$. Then,  because of the relations in the dihedral group, we have  $AA^{-1} = A^{-1}A$,  $AB = BA^{-1}$ and $B^{-1} = B$. Therefore \begin{align*}
SS &= AA \cup BB \cup AB \cup BA;\\
SS^{-1} &= AA^{-1} \cup BB \cup AB;\\
S^{-1}S &= AA^{-1} \cup BB \cup BA;\\
T(S) &= A \cup B \cup AA \cup AA^{-1} \cup BB \cup AB \cup BA \\ &= S \cup SS \cup AA^{-1}.
\end{align*}
We also  note that $\sqrt S = \sqrt A \subseteq \langle x \rangle$. 
\end{obs}

\begin{prop}\label{P1}
Let $n$ be an odd integer, with $n \geq 13$. Then $D_{2n}$ is not filled.
\end{prop}
\begin{proof} Let $n$ be an odd integer with $n \geq 13$ and $G \cong D_{2n}$. Then there is an odd number $k$ for which $n$ is either $5k-6$, $5k-4$, $5k-2$, $5k$ or $5k+2$. 

 Suppose first that $n$ is $5k-2$ for an odd integer $k$. Since $n \geq 13$, we note that $k \geq 3$. Now consider the following subset $S$ of $G$: $$S=\{x^k,x^{k+2}, \cdots, x^{3k-2}; y, xy, \cdots, x^{k-1}y \}.$$ 
 We calculate that  
 $A(SS)=\{x^{2k},x^{2k+2},\cdots,x^{5k-3} \} \cup \{1,x,\cdots,x^{k-1}\} \cup \{x^{4k-1},x^{4k},\cdots,1\}$ and $B(SS)=\langle x \rangle y - B(S)$. Observe that $x^{3k}\notin S \cup SS$; so $S$ does not fill $G$.
 
 Let $A=A(S)$. Then $AA^{-1}=\{1,x^2,x^4,\cdots,x^{2k-2}\} \cup \{x^{3k},x^{3k+2}, \cdots,x^{5k-4},1\}$. Thus $T(S)=G$. By Theorem \ref{T1}(i) therefore, $S$ is locally maximal product-free in $G$, but we have noted that $S$ does not fill $G$.
 
 Next we suppose $n = 5k$ for an odd integer $k$, and again since $n \geq 13$, we have $k \geq 3$. Taking the same set $S = \{x^k,x^{k+2}, \cdots, x^{3k-2}; y, xy, \cdots, x^{k-1}y \}$ we find that $S$ is locally maximal product-free but does not fill $G$.
 
 Now suppose $n=5k+2$ for $k\geq 3$ and odd. The set $U$ given by $$U=\{x^{k-2},x^k, \cdots, x^{3k-2}; y, xy, \cdots, x^{k-3}y \}$$ is locally maximal product-free in $G$ (again using Theorem \ref{T1}(i)), but does not fill $G$ since for example $x^{3k}\notin U \cup UU$.
 
Next suppose $n=5k-6$ for $k\geq 5$ and odd. Then the set $U$ given by $$V=\{x^k, x^{k+2}\cdots, x^{3k-2}; y, xy, \cdots, x^{k-1}y \}$$ is a locally maximal product-free set in $G$ that does not fill $G$. 

Finally, consider the case $n=5k-4$ for $k\geq 5$ and odd. The set $W$ given by $$W=\{x^{k-2},x^k, \cdots, x^{3k-4}; y, xy, \cdots, x^{k-3}y \}$$ is a locally maximal product-free set in $G$ which does not fill $G$. We have now covered all possibilities for $n$, and have shown that in each case $D_{2n}$ is not filled.
\end{proof}

\begin{thm}\label{C1}
The only filled dihedral groups are $D_{6}$, $D_{8}$, $D_{10}$, $D_{12}$, $D_{14}$ and $D_{22}$.
\end{thm}
\begin{proof}
Let $G$ be dihedral of order $2n$. The filled groups of order up to 32 were classified in \cite{AH2015}. The only filled dihedral groups of order up to 32 are $D_{6}$, $D_{8}$, $D_{10}$, $D_{12}$, $D_{14}$ and $D_{22}$. It remains to show that if $n > 16$, then $D_{2n}$ is not filled. Suppose $n > 16$.  By Proposition~\ref{P1}, if $n$ is odd then $D_{2n}$ is not filled, so we can assume $n$ is even. We will show by induction on $m$, that if $G \cong D_{4m}$ for some integer $m$ greater than 3, then $G$ is not filled. Note that  by \cite{AH2015} $D_{16}, D_{20}$, $D_{24}$, $D_{28}$ and $D_{32}$ are not filled. So we can assume $m > 8$.

If $G$ is filled, then by Theorem~\ref{T1}(ii), the quotient $G/Z(G)$ of $G$ by its centre must be filled. But $G/Z(G)$ is dihedral of order $2m$. If $m$ is odd, then by Proposition \ref{P1}, and our assumption that $m > 8$, we have $m = 9$ or $m=11$. We know that $D_{18}$ is not filled, so $m = 11$, meaning $G$ is $D_{44}$. However a straightforward calculation shows that $\{x^2,x^5,x^8,x^{18}, x^{21},x^5y, x^{16}y\}$ is locally maximal product-free in  $D_{44}$, but does not fill $D_{44}$. Thus if $m$ is odd, then $G$ is not filled. Suppose $m$ is even, so $m=2t$ for some $t$ with $t > 4$. Inductively $D_{4t}$ is not filled, so $G$ is not filled. This completes the proof.
\end{proof}

\section{Nilpotent Groups}\label{secnil}

In this section we classify the filled nilpotent groups. The bulk of the work involved here is in determining the filled 2-groups, as it will turn out that there are only two filled nilpotent groups that are not 2-groups. 
We briefly recap some notation and facts around extraspecial groups. For a group $G$ we write $G'$ for the derived group (so $G' = [G,G]$) and $\Phi(G)$ for the Frattini subgroup (the intersection of the maximal subgroups of $G$). A 2-group $G$ is extraspecial if $Z(G) = G' = \Phi(G) \cong C_2$. The order of any extraspecial 2-group is an odd power of 2, and there are exactly two nonisomorphic extraspecial 2-groups of order $2^{2n+1}$ for each positive integer $m$. To describe these, recall the construction of a central product. A central product $A \ast B$ is the quotient of the direct product $A \times B$ by a central subgroup of $A$ and $B$. If $G$ is isomorphic to $A \ast B$, then it has normal subgroups which we may identify with $A$ and $B$, such that $[A,B] = 1$ and $A\cap B \leq Z(G)$.  The extraspecial groups of order 8 are $D_8$ and $Q_8$. If $E_1$ and $E_2$ are the extraspecial groups of order $2^{2n-1}$, for $n \geq 2$, then the extraspecial groups of order $2^{2n+1}$ are isomorphic to $E_1 \ast Q_8$ and $E_2 \ast Q_8$. 

Our first result classifies the 2-groups all of whose quotients are elementary abelian. This is relevant because every quotient of a filled group must be filled, and it will turn out that all but finitely many filled 2-groups are elementary abelian.  

\begin{thm}
	\label{thm1} Suppose every proper nontrivial quotient of a finite nontrivial 2-group $G$ is elementary abelian. Then $G$ is either elementary abelian, extraspecial, $C_4$ or of the form $E \ast C_4$ where $E$ is extraspecial and $|G| = 2|E|$.
\end{thm}
\begin{proof}
It is straightforward to see that the only abelian 2-groups all of whose proper nontrivial quotients are elementary abelian are elementary abelian or cyclic of order 4. We may therefore assume that $G$ is nonabelian. 	Standard results from group theory tell us that $G'$ is the smallest normal subgroup of $G$ with abelian quotient, and that, since $G$ is a $p$-group, $\Phi(G)$ is the smallest normal subgroup with elementary abelian quotient. In particular, $G'\leq \Phi(G)$. In this case though, by hypothesis $G/G'$ is elementary abelian. Therefore $\Phi(G) = G'$.
As $G$ is a 2-group, $G$ has a nontrivial centre and hence at least one central involution $z$. If $G = \langle z\rangle$ then we are done. Otherwise by hypothesis $G/\langle z\rangle$ is elementary abelian. Hence $G' \leq \langle z \rangle$. Since $G$ is nonabelian, we have $G' = \langle z \rangle$. Moreover, were $G$ to contain another central involution $z'$, it would follow that $G' = \langle z'\rangle$, a contradiction. Therefore $Z(G)$ contains a unique involution and is therefore cyclic. If $Z(G) = \langle z\rangle$ then we have $Z(G) = G' = \Phi(G)\cong C_2$, which is the definition of extraspecial. 

The remaining case to consider is that $Z(G)$ is cyclic of order greater than 2. Now $G/\langle z \rangle$ is elementary abelian, so the square of every element of $G$ lies in $\langle z \rangle$. Therefore $Z(G)$ is cyclic of order 4, with $\langle z\rangle = G' = \Phi(G)$. Now $\Phi(G)$ is the intersection of the maximal subgroups of $G$, so there must exist a subgroup $N$ of $G$, of index 2, which does not contain $Z(G)$. This forces $N\cap Z(G) = \langle z \rangle$. Writing $Z = Z(G)$, we have $NZ = G$, $N \cap Z \leq Z(G)$, $[N,Z] = 1$. So $G$ is a central product of $N$ and $Z$, with $|G| = 2|N|$. Since $[N,Z] = 1$ we have $Z(N) \leq Z(G)$, and so $Z(N) = \langle z \rangle$. Thus $N$ is not abelian. Hence $N'$ is nontrivial, which forces $N' = \Phi(N) = \langle z \rangle = Z(N)$. In other words, $N$ is extraspecial. That is, $G$ is the required central product of an extraspecial group and a cyclic group of order 4.   
\end{proof}

\begin{lemma}
	\label{lemma8}
	Let $G$ be a group of the form $E\ast C_4$ where $E$ is extraspecial and $|G| = 2|E|$. Then $G$ is not filled. 
\end{lemma}
\begin{proof}
	We may suppose that $G$ has an extraspecial subgroup $E$ of index 2 and that $Z(G)$ is generated by an element $x$ of order 4, where $E \cap Z(G) = \langle x^2\rangle$. Observe that $Z(E) = \langle x^2\rangle$, and $G' = \Phi(G) = \langle x^2\rangle$. Now $\{x^2\}$ is product-free. So there is a locally maximal product-free set $S$ of $G$ with $x^2 \in S$. If an element $g$ of $G$ has order 4, then $g^2 = x^2$, which means $S$ contains no elements of order 4. So $S$ consists of involutions. Thus $S = S^{-1}$ and $G = S \cup SS \cup \sqrt S$. Now $G = E \cup Ex$. If $h \in E$ has order 4, then $(hx)^2 = h^2x^2 = x^4 = 1$. If $h$ has order 2, then $(hx)^2 = x^2$. So $hx$ is an involution if and only if $o(h) = 4$. Thus $S = A \cup Bx$ where $A$ and $B$ are subsets of $E$ such that the elements of $A$ are involutions and the elements of $B$ have order 4. We have $SS = AA \cup ABx \cup BAx \cup BBx^2$. Now $x \notin E$, so if $x \in S \cup SS$, then we must have $x \in Bx \cup ABx \cup BAx$. Hence $1 \in B \cup AB \cup BA$ which implies $1 \in AB \cup BA$. This would mean that there are elements $a$ of order 2 and $b$ of order 4 with $ab = 1$, which is impossible. Therefore $x \notin S \cup SS$. Hence $S$ is a locally maximal product-free set of $G$ which does not fill $G$. Therefore $G$ is not filled. 
\end{proof}

A group $G$ of order $p^m$ is said to be of {\em maximal class} if $m > 2$ and the nilpotence class of $G$ is $m-1$. It is well known (for example see Theorem 1.2 and Corollary 1.7 of \cite{berkovich}) that the 2-groups of maximal class are dihedral, semidihedral and generalised quaternion. Moreover \cite[Theorem 1.2]{berkovich} if $G$ is a 2-group of maximal class of order at least $16$, then $G/Z(G)$ is dihedral of order $\frac{1}{2}|G|$.  A detailed examination of locally maximal product-free sets in 2-groups of maximal class, which among other things results in an alternative proof of Lemma \ref{maxclass}, appears in \cite{chimere}. However, since we need the result here we thought it would be useful to include a short proof for ease of reference. 

\begin{lemma}\label{maxclass}
	The only filled 2-group of maximal class is $D_8$.	
\end{lemma}
\begin{proof}
	We see from Theorem \ref{C1} that $D_8$ is the only filled dihedral 2-group. It is easy to check that there are no other filled 2-groups with maximal class and order 8 or 16 (in fact by \cite{AH2015} the only nonabelian filled group of order 8 is $D_8$ and the only nonabelian filled group of order 16 is $D_8 \times C_2$). If $G$ is maximal class of order 32 or above, then inductively $G/Z(G)$ is a non-filled dihedral group of order $\frac{1}{2}|G|$, meaning $G$ is not filled. 
\end{proof}

For a $p$-group $G$, we define $c_n(G)$ to be the number of subgroups of $G$ of order $p^n$.
\begin{thm}
	[Theorem 1.17 of \cite{berkovich}] Suppose a 2-group $G$ is neither cyclic nor of maximal class. Then $c_1(G) \equiv 3 \mod{4}$ and for $n > 1$, $c_n(G)$ is even. 
\end{thm}

\begin{cor}
	\label{cor9} Suppose $G$ is a filled group of order $2^n$, where $n > 1$. If the only filled groups of order $2^{n-1}$ are elementary abelian or extraspecial, then $G$ is either elementary abelian, extraspecial or the direct product of a filled extraspecial group of order $2^{n-1}$ with a cyclic group of order 2. If the only filled groups of order $2^{n-1}$ are elementary abelian, then $G$ is either elementary abelian or extraspecial. 
\end{cor}

\begin{proof}
	Note first that if $G$ is a filled 2-group, then $G' = \Phi(G)$. Suppose the only filled groups of order $2^{n-1}$ are elementary abelian or extraspecial. If $n$ is 2 or 3, then the result holds, so we may assume $n \geq 4$. Now $G$ is clearly not cyclic. Moreover, by Lemma \ref{maxclass}, $G$ is not of maximal class.  Therefore $G$ has an even number of subgroups of order 4. The length of any conjugacy class of subgroups of order 4 is either 1 or even. The composition factors of any 2-group are cyclic of order 2, and hence $G$ has at least one normal subgroup of order 4. Therefore $G$ has at least two normal subgroups, $H$ and $K$ say, of order 4. Any nontrivial normal subgroup intersects the centre of $G$ nontrivially, and so $H$ contains a central involution $z$. The quotient $G/\langle z \rangle$ is filled of order $2^{n-1}$ and so, by hypothesis, either elementary abelian or extraspecial. Hence $G/H$, which is isomorphic to $\frac{G/\langle z \rangle}{H/\langle z \rangle}$, is a nontrivial quotient of an extraspecial or elementary abelian 2-group, and is therefore elementary abelian. Similarly $G/K$ is elementary abelian. This implies that $G' = \Phi(G) \leq H \cap K$. If $G$ is abelian, then $G$ is elementary abelian. 
		
		If $G$ is nonabelian, then $G' = \Phi(G) = \langle z \rangle$, where $z$ is a central involution. If $Z(G)$ contains an involution $t$ other than $z$, then since $t$ is not contained in $\Phi(G)$, there is a maximal subgroup $N$ which does not contain  $t$. Thus $G \cong N \times \langle t \rangle$. Now $G/\langle t \rangle \cong N$, which forces $N$ to be filled of order $2^{n-1}$. So $G$ is elementary abelian unless there is a filled extraspecial group $E$ of order $2^{n-1}$, in which case we also have the possibility that $G \cong E \times C_2$. We now deal with the case that $z$ is the only central involution. In that case, since every nontrivial normal subgroup intersects $Z(G)$ nontrivially, every nontrivial normal subgroup contains $z$ and hence every proper quotient is elementary abelian. Therefore, by Theorem \ref{thm1} and Lemma \ref{lemma8}, $G$ is either elementary abelian or extraspecial. 
		
		We have shown that $G$ is either elementary abelian or extraspecial, except in the case where there is a filled extraspecial group $E$ of order $2^{n-1}$, in which case we have the further possibility that $G \cong E \times C_2$.
\end{proof}

\begin{lemma}\label{lem1}
	Let $S$ be a locally maximal product-free set in a group $G$. If $a \in S$ but $a^{-1} \notin S$, then $a^{-1} \in SS \cup \sqrt S$.
\end{lemma}
\begin{proof}
	Since $S$ is a locally maximal product-free set, we have $G = S \cup SS \cup SS^{-1} \cup S^{-1}S \cup \sqrt S$. By assumption $a^{-1} \notin S$. If $a^{-1} \in SS^{-1}$, then there are $b, c \in S$ with $a^{-1} = bc^{-1}$. But this implies $ab = c$, contradicting the fact that $S$ is product-free. If $a^{-1} = b^{-1}c \in S^{-1}S$, then $b = ca$, another contradiction. Therefore $a^{-1} \in SS \cup \sqrt S$.
\end{proof}

\begin{lemma}
	\label{involutions} Suppose $G$ is a group of exponent 4 all of whose elements of order 4 square to the same central involution $z$. If $S$ is a locally maximal product-free set that does not fill $G$, then $S$ contains $z$ and every element of $S$ is an involution. 
\end{lemma}
\begin{proof}
	Suppose some element $a$ of $S$ is not an involution, and assume for a contradiction that $a^{-1} \notin S$. Now $a^2 = z$, where $z$ is the central involution of $G$. Hence $z \notin S$. Therefore $a^{-1} \notin \sqrt S$. Hence, by Lemma \ref{lem1}, $a^{-1} \in SS$. This means $a^{-1} = bc$ for some $b, c, \in S$. Now if $b$ is an involution then, taking the inverse of both sides, we have $a = c^{-1}b$, which forces $ca = b$, contradicting the fact that $S$ is product-free. Therefore $b$ has order 4, which implies that $a^2 = b^2 = z$. From $za^{-1} = zbc$ we then get $a = b^{-1}c$, whence $ba = c$, another contradiction. Therefore our assumption was false, and $a^{-1} \in S$. This is true for all non-involutions in $S$, and so $S^{-1} = S$. Because $S$ is locally maximal product-free, we know $S \cup SS \cup SS^{-1} \cup S^{-1}S \cup \sqrt S = G$. But $\sqrt S  = \emptyset$ because $z \notin S$, and $SS^{-1} = SS = S^{-1}S$ because $S = S^{-1}$. Therefore $G = S\cup SS$ and $S$ fills $G$.
	
	 We have shown so far that if $S$ is locally maximal and does not fill $G$, then every element of $S$ must be an involution. If $S$ consists of involutions, then $S = S^{-1}$; so $S \cup SS \cup SS^{-1} \cup S^{-1}S = S \cup SS$. If $\sqrt S = \emptyset$ then we would have $S \cup SS  = G$, meaning $S$ fills $G$. Hence $\sqrt S$ cannot be empty. Thus $S$ contains $z$ and every element of $S$ is an involution.
\end{proof}

\begin{prop}\label{computercheck}
If $G$ is a non-abelian filled 2-group of order up to 128, then $G$ is either $D_8$, $D_8 \times C_2$, $D_8 \ast Q_8$ or $(D_8 \ast Q_8) \times C_2$.
\end{prop}

\begin{proof}
Computer search allows us to show that the only nonabelian filled 2-groups of order up to 32 are $D_8$, $D_8 \times C_2$ (fitting in with Corollary \ref{cor9}) and $D_8 \ast Q_8$. Corollary \ref{cor9} tells us that the only candidates for filled groups of order 64 are $C_2^6$ and $(D_8 \ast Q_8) \times C_2$. Lemma \ref{involutions} allows us to reduce the work involved in checking that $(D_8 \ast Q_8) \times C_2$ is filled, by checking only sets of involutions. By this means, it is then possible to check by machine that $(D_8 \ast Q_8) \times C_2$ is indeed the only filled non-abelian group of order 64. By restricting the search to non-abelian groups whose quotients are filled and looking only at product-free sets consisting of involutions, computer search also confirmed that there are no non-abelian filled groups of order 128. See Section \ref{secalg} for more details on the algorithms used.
\end{proof}

\begin{notation} \label{notn} Let $G$ be extraspecial of order greater than $128$, so that $|G| = 2^{2n+5}$ for some $n > 1$. Then $G$ has subgroups $H_1, \ldots, H_n$ and $Q$, all isomorphic to $Q_8$, and a subgroup $K$ isomorphic to either $D_8$ or $Q_8$, such that $$G = KH_1\cdots H_nQ$$ where $[K,H_i] = [H_i,H_j] = [H_i,Q] = [K,Q] = 1$ for all distinct $i$, $j$. Furthermore there is an involution $z$ of $G$ such that for all distinct $i, j$ we have $K \cap H_i = H_i \cap H_j = H_i \cap Q = K \cap Q  = \langle z \rangle = Z(G) \cong C_2$.

Write $E = K H_1  \cdots H_n$, so that $E$ is an extraspecial subgroup of index 4 in $G$. We may write
$H_i = \langle a_i, b_i : a_i^4 = 1, b_i^2 = a_i^2, b_ia_i = a^{-1}_ib_i\rangle$, 
$Q = \langle a, b : a^4 = 1, b^2 = a^2, ba = a^{-1}b\rangle$
and $K = \langle \alpha, \beta\rangle$ where $\alpha^4 = 1$, $\beta\alpha = \alpha^{-1}\beta$. If $K \cong D_8$ then $\beta^2 = 1$; otherwise $\beta^2 = \alpha^2$. Note that $z = \alpha^2 = a_1^2 = \cdots = a_n^2 = a^2$.
Elements $g$ of $G$ can be written canonically as $g = dh_1 \cdots h_nq$ where $h_i \in \{1, a_i, b_i, a_ib_i\}$, $q \in \{1, a, b, ab\}$ and $d \in K$. 
Observe that an element $g$ of $G$ has order 4 if and only if an odd number of $d, h_1, \ldots, h_n$ and $q$ have order 4. 
\end{notation}

The next theorem will show that there are no extraspecial filled groups of order greater than 32. The method, using the notation just described, is to break $G$ into the four cosets $E$, $Ea$, $Eb$ and $Eab$. Then we form a set $S$ by taking the union of a suitable locally maximal product-free set of involutions of $K$, with (a carefully chosen) half of the involutions in $Ea$ and half of the involutions in $Eb$. We show that this set $S$ is product-free and then that $S$ is locally maximal, but does not fill $G$. (In particular, $S \cup SS$ cannot contain either $a$ or $b$.) 

\begin{thm}\label{exsp}
If $G$ is an extraspecial group of order greater than 128, then $G$ is not filled.
\end{thm}

\begin{proof} Let $G$ be extraspecial of order greater than 128. Writing $G$ as described in Notation \ref{notn}, let $$U = \{g \in E: o(g) = 4 {\text{ and }} g   = dh_1\cdots h_{n} {\text{ with }} d \in \{1, \alpha, \beta, \alpha\beta\}\}.$$ Notice that for any $g \in U$ we have $g^{-1} = zg$, and so $U \cap U^{-1} = \emptyset$. Moreover $U \cup U^{-1}$ comprises every element of order 4 in $E$. We next define a certain subset $X$ of $K$ as follows.
	$$X = \left\{\begin{array}{ll}\{z, z\beta, z\alpha\beta\} & {\text{ if $K \cong D_8$;}}\\
\{z\} & {\text{ if $K \cong Q_8$.}}	\end{array}\right.$$  In each case $X$ is a locally maximal product-free set of involutions in $K$. Finally, let $$S = X \cup Ua \cup Ub.$$ We claim that $S$ is a locally maximal product-free set of $G$ that does not fill $G$. We have

$$SS = (XX\cup UUz) \cup (XU \cup UX)a \cup (XU \cup UX)b \cup (UU \cup UUz)ab.$$

Our first goal is to show that $S$ is product-free. We do this by noting that $G = E \cup Ea \cup Eb \cup Eab$ and considering the intersection of $S\cap SS$ with each of these cosets in turn. Consider $(S\cap SS) \cap E$. This is equal to $X \cap (XX \cup UUz)$. Now $X$ is locally maximal product-free in $K$. Thus $X \cap XX = \emptyset$.  Suppose $x \in X \cap UUz$. Then there are elements $g, g'$ of $U$ with $gg'z = x$, or equivalently $gg'zx = 1$. If $x = z$ then $g' = g^{-1}$, but $U \cap U^{-1} = \emptyset$, a contradiction. So $x \neq z$. This implies we are in the case $K \cong D_8$, and furthermore that $x \in \{z\beta, z\alpha\beta\}$.  Writing $g = dh_1 \cdots h_n$ and $g' = d'h_1'\cdots h_n'$ in the canonical way we have that $gg'zx = 1$, and so $(dd'zx)(h_1h_1')\cdots (h_nh_n') = 1$. Thus $hh_i' \in \{1, z\}$ for all $i$. Since $\{h_i, h_i'\} \subseteq \{1, a_i, b_i, a_ib_i\}$, this implies $h_i = h_i'$ for all $i$. Thus $g = dh$ and $g' = d'h$, where $h = h_1\cdots h_n$. As they are elements of $U$, both $g$ and $g'$ must have order 4, which implies that $d$ has order 4 if and only if $d'$ has order 4. Since we are in the case $K \cong D_8$ we note that $\alpha$ has order 4 and $\alpha^j\beta$ has order 2 for all $j$. Now $\{d, d'\} \subseteq \{1, \alpha, \beta, \alpha\beta\}$. So either $d = d' = \alpha$ or $\{d, d'\}\subseteq \{1, \beta, \alpha\beta\} = zX$. In the former case, $g = g'$ so $gg'z = 1 \notin X$. In the latter case, $h$ must have order 4 (otherwise $g$ would be an involution) and so $gg'z = dd'h^2z = dd' \in (zX)(zX) = XX$. But $X$ is product-free, so $gg'z \notin X$. Therefore $(SS \cap S)\cap E = \emptyset$.

Next we look at $(S \cap SS) \cap Ea$. This is equal to $((XU \cup UX) \cap U)a$. Suppose there are $x$ in $X$ and $u$ in $U$ such that $xu \in U$ or $ux \in U$. If $x = z$ then $xu = ux = u^{-1} \notin U$, so again we are reduced to the case $K \cong D_8$ and $x \in \{z\beta, z\alpha\beta\}$. We can write $u = dh_1 \cdots h_n$ where $d \in \{1, \alpha, \beta, \alpha\beta\}$ and $h_i \in \{1, a_i, b_i, a_ib_i\}$. Let $h = h_1\cdots h_n$. If $d = \alpha$ then $u = \alpha h$, $ux = (\alpha x)h$ and $xu = (x\alpha) h$. Since $u$ has order 4, it follows that $h^2 = 1$. But $\alpha x \in \{z\alpha\beta, \beta\}$ and $x \alpha \in \{\alpha\beta, z\beta\}$, whence $ux$ and $xu$ both have order 2 and so are not contained in $U$, a contradiction. Therefore $d$ is contained in $\{1, \beta, \alpha\beta\} = zX$. In particular $d$ is an involution, which implies $h$ has order 4. From the fact that $d \in zX$ we get that $dx \in zXX$, which intersects $zX$ trivially as $X$ is product-free. Thus $dx \in K\setminus zX = \{z, \alpha, z\alpha, z\beta, z\alpha\beta\}$. Thus either $dx \notin \{1, \alpha, \beta, \alpha\beta\}$, meaning $(dx)h$ does not have the right first component to be an element of $U$, or $dx = \alpha$ which means $dxh$ has order 2. The same argument applies to $xd$. Thus, either way, $xu$ and $ux$ are not elements of $U$. Thus $(S\cap SS) \cap Ea = \emptyset$. Since $(S\cap SS) \cap Eb = ((XU \cup UX) \cap U)b$, the same argument shows that $(S\cap SS)\cap Eb = \emptyset$. Finally, $S \cap Eab = \emptyset$, so clearly $(S\cap SS) \cap Eab = \emptyset$. Therefore we have shown that $S \cap SS = \emptyset$. Hence $S$ is product-free.\\

In order to show that $S$ is locally maximal, we need to examine $UUz$ more carefully. We will show that $E\setminus K \subseteq UUz$. Let $g \in E$, and write (canonically) $g = dh_1 \cdots h_n$ where $d \in K$ and $h_i \in \{1, a_i, b_i, a_ib_i\}$. Suppose there is some $i$ for which $h_i \neq 1$.  Without loss of generality we can assume $i=1$. 
Suppose $g$ has order 4. We use the value of $h_1$ to define elements $v_1$ and $w_1$ of $H_1$ as follows.

\begin{center}
	\begin{tabular}{c|cc}
		$h_1$ & $v_1$ & $w_1$ \\
		\hline  $a_1$ & $a_1b_1$& $b_1$\\
		  $b_1$ & $a_1$& $a_1b_1$\\
		  $a_1b_1$ & $b_1$& $a_1$
	\end{tabular}
\end{center}
 In each case $v_1w_1 = h_1z$ and $w_1v_1 = h_1$. Now set $g' = dh_1' h_2\cdots h_n$. Exactly one of $g'$ and $g' z$ is in $U$, and also $w \in U$. Now $g' w_1 z = w_1 (zg')z = g$. Exactly one of these expressions is in $UUz$. Thus $g \in UUz$.  

Now suppose $g$ has order 2, and we are still assuming $g = dh_1\cdots h_n$ where $h_1 \neq 1$ and $d \in K = \{1, \alpha, \beta, \alpha\beta, z, z\alpha, z\beta, z\alpha\beta\}$. If $d \in \{z, z\alpha, z\beta, z\alpha\beta\}$ then let $u_1 = h_1$ and $u_2 = (zd)h_2\cdots h_n$. Then $u_1$, $u_2 \in U$ and $g = u_1u_2z \in UUz$. If $d = 1$ then $\alpha, \alpha g \in U$ and $g = (\alpha)(\alpha g)z \in UUz$. Suppose $d \in \{\alpha, \beta,\alpha\beta\}$ and that $d$ has order 4. If $h_j = 1$ for any $j$, then without loss of generality $h_2 = 1$; set $u_1 = a_2g$ and $u_2 = a_2$. Then $u_1, u_2 \in U$ and $g = u_1u_2z \in UUz$. On the other hand, if for all $j$ we have $h_j \neq 1$, then for each $h_j$ there are $v_j$ and $w_j$ in $H_j$ (as in the table above for $h_1$) satisfying $v_jw_j = h_jz$ and $w_jv_j = h_j$. Moreover, from the assumption that $d$ has order 4 and $g$ has order 2, it follows that $n$ is odd; in particular $n \geq 3$, since we are assuming $n>1$. 
Set $u_1 = h_1w_2\cdots w_{n-1}v_n$ and $u_2 = dv_2\cdots v_{n-1}w_n$. Then since $n$ is odd, $u_1$ and $u_2$ have order 4 and so $u_1, u_2 \in U$. Also $u_1u_2z = d(h_1)(w_2v_2)\cdots (w_{n-1}v_{n-1})\cdots (v_nw_nz) = g$. Thus $g \in UUz$. Finally we must consider the case where $g$ has order 2, $d \in \{\alpha, \beta, \alpha\beta\}$ and $d$ has order 2. This means $K \cong D_8$, $d \in \{\beta, \alpha\beta\}$ and an even number of the $h_i$ are nontrivial (including, by assumption, $h_1$). If $d = \beta$, set $u_1 = \alpha w_1 h_2\cdots h_n$ and $u_2 = \alpha\beta v_1$. If $d = \alpha\beta$, set $u_1 = \alpha v_1h_2\cdots h_n$ and $u_2 = \beta w_1$. Then $u_1, u_2 \in U$ and $g = u_1u_2z \in UUz$.  We have now shown that $E\setminus K \subseteq UUz$. \\

To show that $S$ is locally maximal, we note first that since $S$ contains $z$, all elements of order 4 in $G$ are contained in $\sqrt S$. Therefore $S$ is locally maximal if and only if every involution of $G$ is contained in $S \cup SS$. Involutions of $Ea$ are of the form $ga$ where $g \in E$ and $g$ has order 4. The set of elements of order 4 in $E$ is $U \cup Uz$. Therefore the involutions of $Ea$ are contained in the set $Ua \cup UXa$ which is contained in $S \cup SS$. Similarly every involution in $Eb$ appears in $S \cup SS$. 

Now consider $(S \cup SS) \cap E$, which is given by $X \cup XX \cup UUz$. Every involution of $K$ is contained in $X \cup XX$, because $X$ fills $K$, and every involution of $E \setminus K$ is contained in $UUz$. Thus $S \cup SS$ contains all the involutions of $E$. Finally we look at $(S \cup SS) \cap Eab$, which is the set
$(UU \cup UUz)ab$. The involutions of $Eab$ are elements $gab$ where $g$ in $E$ has order 4. Since $E \setminus K \subseteq UUz$, all that remains is to express every element of $K$ that has order 4 as an element of $UU$ or $UUz$. If $K \cong D_8$, then since $n > 1$ we can write $\alpha = (\beta a_1)(\alpha\beta a_1) \in UU$, and so $\alpha z \in UUz$, and we are done. If $K \cong Q_8$ then we have $\alpha = \beta(\alpha\beta)$, $\beta = (\alpha\beta)\alpha$ and $\alpha\beta$ as elements of $UU$, and their inverses as elements of $UUz$. Therefore every involution of $Eab$ is indeed contained in $S \cup SS$. We have now shown that $S\cup SS$ contains all the involutions of $G$, and hence $S$ is locally maximal product-free. However, $a \notin S \cup SS$. Therefore $S$ does not fill $G$. Thus $G$ is not a filled group. 
\end{proof}

\begin{cor}\label{2group}
Let $G$ be a 2-group. Then $G$ is filled if and only if $G$ is either elementary abelian, or one of $D_8$, $D_8 \times C_2$, $D_8 \ast Q_8$ or $(D_8 \ast Q_8) \times C_2$.
\end{cor}

\begin{proof}
The proof is immediate from Corollary \ref{cor9}, Proposition \ref{computercheck} and Theorem \ref{exsp}.
\end{proof}

\begin{thm}\label{nilpotent} Let $G$ be a finite nilpotent group. Then $G$ is filled if and only if $G$ is either an elementary abelian 2-group or one of $C_3, C_5$, $D_8$, $D_8 \times C_2$, $D_8 \ast Q_8$ or $(D_8 \ast Q_8) \times C_2$.\end{thm}

\begin{proof} 
Suppose $G$ is filled and nilpotent. Then $G$ is the direct product of its Sylow subgroups. Therefore for any prime $p$ dividing $|G|$, $G$ has a normal subgroup $N$ of index $p$. Hence, by Theorem \ref{T1}(ii) and (iii), $p$ is one of $2, 3$ or $5$. If $p = 3$, then by Theorem \ref{T1}(iv), $G$ must be cyclic of order 3. So we can assume the only primes dividing $|G|$ are 2 and 5. If $p = 5$ and 25 divides $|G|$ then $G$ has a normal subgroup of index 25, but by Theorem \ref{T1}(iii) there are no filled groups of order 25, a contradiction. Therefore the normal subgroup $N$ of index 5 in $G$ is either trivial or a 2-group. Either way, $N$ contains no elements of order 5. Hence, by Theorem \ref{T1}(v), $G$ must be cyclic of order 5. The only remaining possibility is that $G$ is a 2-group. Theorem \ref{nilpotent} now follows from Corollary \ref{2group}.\end{proof}

\section{Groups of order $2^np$}\label{sec2np}

In this section we show that if $G$ is a group of order $2^np$, where $n$ is a positive integer and $p$ is an odd prime, 
then $G$ is filled if and only if $G$ is $D_6$, $D_{10}$, $D_{12}$, $D_{14}$ or $D_{22}$. 

\begin{lemma}\label{lem:2kp}
Let $p$ be an odd prime and let $k$ be an integer satisfying $\displaystyle k>\sum_{r=1}^{\infty}\left\lfloor\frac{p}{2^r}\right\rfloor$. 
Let $G$ be a group of order $2^k p$.
Then $G$ contains a non-trivial normal elementary abelian 2-subgroup of order no greater than $2^p$.
\end{lemma}
\begin{proof}
We show first that $G$ contains some non-trivial normal 2-subgroup $N$.
Consider the set $S_2$ of Sylow 2-subgroups of $G$. By the Sylow theorems, either $|S_2|=1$ or $|S_2|=p$.
If $|S_2|=1$ we take $N$ to be the unique Sylow 2-subgroup.
If $|S_2|=p$ then $G$ acts transitively by conjugation on the set $S_2$, and so the kernel $N$ of this action is a normal subgroup of $G$ which is a 2-group.
The condition on $k$ ensures that $G$ is sufficiently large that $N$ must be non-trivial.

It is a fundamental result that a minimal normal subgroup of a solvable group is elementary abelian. 
Thus $N$ contains some non-trivial elementary abelian 2-subgroup $K$ which is normal in $G$. Now $K$ is a union of conjugacy classes of $G$. 
Since $|K|$ is even and contains the conjugacy class $\{1\}$, it must contain some other conjugacy class $T$ of odd length.
Since $|T|$ must divide $|G|$, we conclude that either $|T|=1$ or $|T|=p$. 
In either case, $\langle T\rangle$ is a normal 2-subgroup of $G$ of order at most $2^p$, as required.
\end{proof}

\begin{cor}\label{n=3}
For any $k\geq 3$, there is no filled group of order $3\times 2^k$.
\end{cor}
\begin{proof}
We proceed by induction. By computer search (see Section \ref{secalg} for details) we know there are no filled groups of order 24, 48 or 96. So the statement is true for $k=3,4,5$.
Suppose the statement is true up to $k\geq 5$ and consider the case $k+1$.
If $G$ is a group of order $3\times 2^{k+1}$, then by Lemma~\ref{lem:2kp} it contains a normal subgroup $H$ of order 2, 4 or 8.
Then $G/H$ has order $3\times 2^{k-2}$, $3\times 2^{k-1}$ or $3\times 2^k$ and so is not filled by the induction hypothesis. 
Thus $G$ is not filled.
\end{proof}

\begin{cor}\label{n=5}
For any $k\geq 2$, there is no filled group of order $5\times 2^k$.
\end{cor}
\begin{proof}
By computer search  we know there are no filled groups of order 20, 40, 80, 160 or 320. So the statement is true for $k=2,3,4,5,6$.
Suppose the statement is true up to $k\geq 6$ and consider the case $k+1$.
If $G$ is a group of order $5\times 2^{k+1}$, then by Lemma~\ref{lem:2kp} it contains a normal subgroup $H$ of order 2, 4, 8, 16 or 32.
Then $G/H$ has order $5\times 2^{k-4}$, $5\times 2^{k-3}$, $5\times 2^{k-2}$, 
$5\times 2^{k-1}$ or $5\times 2^k$ and so is not filled by the induction hypothesis. Thus $G$ is not filled.
\end{proof}

\begin{cor}\label{n=7}
For any $k\geq 2$, there is no filled group of order $7\times 2^k$.
\end{cor}
\begin{proof}
By computer search, we know there are no filled groups of order 28, 56, 112, 224, 448, 896 or 1792. So the statement is true for $k=2,3,4,5,6,7,8$.
Suppose the statement is true up to $k\geq 8$ and consider the case $k+1$.
If $G$ is a group of order $7\times 2^{k+1}$, then by Lemma~\ref{lem:2kp} it contains a normal subgroup $H$ of order 2, 4, 8, 16, 32, 64 or 128.
Then $G/H$ has order $7\times 2^{k-6}$, $7\times 2^{k-5}$, $7\times 2^{k-4}$, $7\times 2^{k-3}$, 
$7\times 2^{k-2}$, $7\times 2^{k-1}$ or $7\times 2^k$ and so is not filled by the induction hypothesis. 
Thus $G$ is not filled.
\end{proof}

\begin{lemma} \label{ncyclic} Suppose $G$ is a filled group of order $2^np$, where $n \geq 2$ and $p$ is an odd prime. If $G$ has a normal subgroup of order $p$, then $G$ contains a central involution.  
\end{lemma}

\begin{proof}
Suppose $N$ is normal of order $p$ in $G$. Then $G = NH$ where $H$ is any Sylow 2-subgroup of $G$. This means $G/N \cong H$. Since $G$ is filled, $G/N$ must be filled. By Corollary \ref{2group} $H$ is either an elementary abelian 2-group, or $D_8$, $D_8\times C_2$, $D_8\ast Q_8$ or $D_8\ast Q_8 \times C_2$. Since $H$ has order at least 4, it follows that either $H$ contains a Klein 4-group $K = \langle a, b\rangle$ such that $K$ is central in $H$, or $H$ contains a subgroup $D$ which is dihedral of order 8, whose centre is also the centre of $H$. 
In the first scenario, consider the action of $H$ on $N$ by conjugation. Write $N = \langle x \rangle$. Now $axa^{-1} = x^i$ for some $i$, and $x = a^2xa^{-2} = x^{i^2}$. Thus $i = \pm 1$ (because in the cyclic group of units of $\mathbb{Z}_p$ the element 1 has exactly 2 square roots). If $axa^{-1} = x^{-1}$ and $bxb^{-1} = x^{-1}$, then $(ab)x(ab)^{-1} = x$. Therefore at least one involution $g$ in $K$ centralises $x$. This means we have $g \in Z(H) \cap C_G(N)$. Thus $g \in Z(G)$. Now consider the second situation, where $H$ contains a subgroup $D$ which is dihedral of order 8 whose centre is also the centre of $H$. We have $D = \langle r,s: r^2 = s^2 = (rs)^4 = 1\rangle$. Again looking at the action on $N$ by conjugation, we have that $rxr^{-1}  = x^{\pm 1}$ and $sxs^{-1} = x^{\pm 1}$, which implies $(rs)x(sr)^{-1} = x^{\pm 1}$. Let $g = (rs)^2$. Then $g \in Z(H)$ and $gxg^{-1} = x$, so $g \in C_G(N)$. Hence again $G$ contains a central involution.
	\end{proof}

\begin{prop}
	\label{n=11} For any $k \geq 2$, there is no filled group of order $11 \times 2^k$.
\end{prop}

\begin{proof} We proceed by induction on $k \geq 2$. Computer search shows there is no filled group of order 44. Let $G$ be a group of order $11\times 2^k$ for $k > 2$ and suppose for a contradiction that $G$ is filled. If $G$ has a normal Sylow 2-subgroup then the quotient of $G$ by this subgroup would be filled of order 11, which is impossible. So we can assume $G$ does not have a normal Sylow 2-subgroup. If $G$ has a normal Sylow 11-subgroup $N$, then, by Lemma \ref{ncyclic}, $G$ contains a central involution $g$. The quotient $G/\langle g \rangle$ is filled of order $11 \times 2^{k-1}$. By induction $G/\langle g \rangle$ is not filled, and so $G$ cannot be filled. Suppose then that the Sylow subgroups are not normal. The number of Sylow 11-subgroups divides $2^k$ and is congruent to 1 modulo 11. So the first time this can arise is when $k=10$. A simple counting argument shows that any group of order $11 \times 2^{10}$ has either a normal Sylow 11-subgroup or a normal Sylow 2-subgroup, so there is nothing to check here. There is one group of order $2^{11} \times 11$ with non-normal Sylow subgroups, and four such groups of order $2^{12}\times 11$. The package GrpConst in GAP \cite{GAP4} allows the user to construct all solvable groups of given order, and the function FrattiniExtensionMethod restricts to those groups with only non-normal Sylow subgroups. Thus, even though these five groups are not contained in the Small Groups library of GAP, they can be constructed and tested using the methods described in Section \ref{secalg}. The upshot is that no group of order $11 \times 2^{10}$, $11 \times 2^{11}$ or $11 \times 2^{12}$ is filled. We may therefore assume $k \geq 13$. By Lemma \ref{lem:2kp} there is a normal elementary abelian 2-subgroup $N$ of $G$ with order at most $2^{11}$. Thus $G/N$ is filled of order $11 \times 2^m$ where $2 \leq m < k$, a contradiction. Hence $G$ is not filled. The result now follows by induction.\end{proof}

\begin{thm}
	Let $G$ be a group of order $2^np$ where $n \geq 1$ and $p$ is an odd prime. If $G$ is filled, then $G$ is one of $D_6, D_{10}$, $D_{12}$, $D_{14}$ or $D_{22}$.
\end{thm}
\begin{proof}
We have dealt with $p = 3, 5, 7$ and $11$. It only remains to show that if $p > 11$, then there are no filled groups of order $2^np$. We proceed by induction on $n$. If $n=1$, then the result holds by Theorems~\ref{T1}(iii) and \ref{C1}. Suppose $n \geq 2$. Let $N$ be a minimal normal subgroup of $G$. Then $N$ is either cyclic of order $p$ or an elementary abelian 2-group. If $N$ is cyclic of order $p$ then by Lemma \ref{ncyclic},  $G$ has a central involution $g$. Now $G/\langle g \rangle$ has order $2^{n-1}p$, so by inductive hypothesis is not filled. Hence $G$ is not filled. So assume $N$ is an elementary abelian 2-group. Then $G/N$ is either cyclic of order $p$ or has order $2^mp$ where $1 \leq m < n$. In either case, since $p > 11$, we know that $G/N$ is not filled. Therefore $G$ is not filled. By induction no group of order $2^np$ is filled, when $p > 11$. This completes the proof. 	
\end{proof}

\section{Groups of order up to 2000}\label{secalg}

In this section we describe the computer algorithms used to determine the filled status of a group.
These algorithms are implemented in GAP~\cite{GAP4} and allow us to test all groups in the library of small groups up to order 2000. We note that although there are nearly 50 billion groups of these orders, the vast majority are accounted for by the 2-groups (classified in Section \ref{secnil}) and groups of order 1536 (classified in Section \ref{sec2np}).

The first algorithm (Algorithm~\ref{alg:nonfill}) attempts to find a locally maximal product-free set in a given group $G$ which does not fill $G$. 
The strategy is to repeatedly add elements at random to a product-free set $S$ until $S$ is locally maximal. 
At each stage we keep track of the set $F$ of elements which could be added to $S$ to keep it product-free.
If our maximal set $S$ fills $G$ we discard it and start again, returning the first set $S$ found which does not fill $G$.
Note that by Lemma~\ref{involutions}, if $G$ is an extraspecial 2-group we may begin each search by placing the unique central involution in $S$.
In practice if this algorithm fails to return a result in a reasonable time we abort and use the exhaustive search method of Algorithm~\ref{alg:exhaustive}.

\begin{algorithm}
\begin{algorithmic}
\caption{Find a non-filling locally maximal product-free set for a group $G$}
\label{alg:nonfill}
\Function{NFS}{$G$}
\Repeat
\If{$G$ is an extraspecial 2-group}
	\State $S\gets Z(G)\setminus\{1\}$
\Else
	\State $S\gets \emptyset$
\EndIf
	\State $F\gets G\setminus (\{1\}\cup S\cup \sqrt{S})$
	\Repeat
		\State $x\gets \mathrm{Random}(F)$
		\State $S\gets S\cup\{x\}$
		\State $F\gets F\setminus (S\cup SS\cup SS^{-1}\cup S^{-1}S\cup \sqrt{S})$
	\Until $F=\emptyset$
\Until $\{1\}\cup S\cup SS\neq G$
\State\Return $S$
\EndFunction
\end{algorithmic}
\end{algorithm}

The second algorithm (Algorithm~\ref{alg:exhaustive}) performs an exhaustive search of locally maximal product-free sets $S$ in a group $G$ and tests whether any fails to fill $G$.
This algorithm is very expensive, and is only required when the random method of Algorithm~\ref{alg:nonfill} has failed to return a result in a reasonable time.
The key to making this algorithm run efficiently is the observation that if $\phi$ is an automorphism of $G$, then $S$ is a locally maximal product-free
subset of $G$ if and only if $\phi(S)$ is locally maximal product-free. Thus the problem of testing all possible sets $S$ is reduced to testing only orbit
representatives under the action of the automorphism group of $G$. 

While these orbits can be readily found using GAP, in practice computing orbits of all possible subsets of $G$ is still prohibitively expensive.
To get round this problem, we begin by computing orbits of all product-free sets $S$ of size 3. 
From each orbit we choose the minimal representative set (with respect to some arbitrary ordering of the elements of $G$).
For each such representative set $S$, we then try to extend $S$ in all possible ways to obtain a locally maximal product-free set and test whether each possible
extension fills our group $G$. 
We need only consider extensions using the set $F$ of elements larger than any currently in our set $S$ and which keep $S$ product-free, so again we keep track of
this set. Each time we add a new element $x$ to $S$, we test whether $S$ is locally maximal and if not, we (recursively) extend this new set.
The algorithm terminates when either a non-filling locally maximal set has been found, or all possible sets have been examined.

\begin{algorithm}
\begin{algorithmic}
\caption{Exhaustive search for locally maximal product-free sets}
\label{alg:exhaustive}
\Function{ExhaustiveSearch}{$G$}
\State $O\gets$ set of orbit representatives of product-free sets of 3 elements of $G$ under action of $\mathrm{Aut}(G)$
\ForEach{$S\in O$}
	\State $F\gets G\setminus (S\cup SS\cup SS^{-1}\cup S^{-1}S\cup \sqrt{S})$
	\If{\Not\Call{ExtendPFS}{$G$,$S$,$F$}}
		\State\Return\False
	\EndIf
\EndFor
\State\Return\True
\EndFunction
\State
\Function{ExtendPFS}{$G$,$S$,$F$}
\If{$F=\emptyset$}
	\If{$S\cup SS\cup SS^{-1}\cup S^{-1}S\cup \sqrt{S}=G$}
		\If{$\{1\}\cup S\cup SS\neq G$}
			\State\Return\False
		\EndIf
	\EndIf
\Else
	\ForEach{$x\in F$}
		\State $S'\gets S\cup\{x\}$
		\State $F'\gets\{f\in F|f>x\}\setminus (S'\cup S'S'\cup S'S'^{-1}\cup S'^{-1}S'\cup \sqrt{S'})$
		\If{\Not\Call{ExtendPFS}{$G$,$S'$,$F'$}}
			\State\Return\False
		\EndIf
	\EndFor
\EndIf
\State\Return\True
\EndFunction
\end{algorithmic}
\end{algorithm}

Our final algorithm (Algorithm~\ref{alg:filled}) is used to determine whether a given group $G$ is filled. It uses the results from previous sections to exclude most groups without
the need to resort to construction of non-filling sets. For those groups which cannot be excluded in this way, we use the random method of 
Algorithm~\ref{alg:nonfill} to find a non-filling set. If all else fails, we resort to the exhaustive method of Algorithm~\ref{alg:exhaustive}.

We begin by defining the set $\mathcal{G}$ of filled groups of order at most 32, as given in~\cite[Table 1]{AH2015}.
For larger groups we then apply the simple tests using Theorems \ref{T1}(iii), \ref{T1}(vi), \ref{T1}(vii), \ref{C1} and \ref{exsp}.
If these are not sufficient to determine the status of our group we examine its normal subgroups and invoke Theorems \ref{T1}(ii), \ref{T1}(iv) and \ref{T1}(v).
Finally, if the status of the group is still not resolved we use Algorithms \ref{alg:nonfill} then \ref{alg:exhaustive} to search for non-filling sets.

\begin{algorithm}[h!]
\begin{algorithmic}
\caption{Test whether a group $G$ is filled}
\label{alg:filled}
\Function{Filled}{$G$}
\State $n\gets|G|$
\If{$n\leq 32$}
	\If{$G\in\mathcal{G}$}
		\State\Return\True
	\Else
		\State\Return\False
	\EndIf
\ElsIf{$n$ is odd}
	\State\Return\False
\ElsIf{$G$ is elementary abelian}
	\State\Return\True
\ElsIf{$n=2^k$ where $k>7$}
	\State\Return\False
\ElsIf{$n=2^k p$ where $k>0$ and $p$ is an odd prime}
	\State\Return\False
\ElsIf{$G$ is abelian, dihedral or generalised quaternion}
	\State\Return\False
\Else
	\ForEach{proper non-trivial normal subgroup $N\lhd G$}
		\If{$[G:N]=3$ or $[G:N]=5$ and not all elements of order 5 are in $N$}
			\State\Return\False
		\EndIf
		\If{\Not \Call{Filled}{$G/N$}}
			\State\Return\False
		\EndIf
	\EndFor
\EndIf
\If{\Call{NFS}{$G$} succeeds}
	\State\Return\False
\Else
	\State\Return\Call{ExhaustiveSearch}{$G$}
\EndIf
\EndFunction
\end{algorithmic}
\end{algorithm}

We remark that when we were proving the results of Section \ref{sec2np} about groups of order $2^kp$, we used a version of Algorithm~\ref{alg:filled} that omits the condition\smallskip\\
\hspace*{1in} {\bf{else if }} $n=2^k p$ where $k>0$ and $p$ is an odd prime {\bf then}\\
\hspace*{1in} {\bf{return false}}.\smallskip

 Using these methods we have examined all groups in the small groups library in GAP up to order 2000. 
The only filled groups are those noted in~\cite[Table 1]{AH2015} plus the group $(D_8\ast Q_8)\times C_2$ of order 64 and the elementary abelian 2-groups. We conclude with the following conjecture.

\begin{conj}
 Let $G$ be a finite group. Then $G$ is filled if and only if $G$ is either an elementary abelian 2-group or one of $C_3, C_5$, $D_6$, $D_8$, $D_{10}$, $D_{12}$,  $D_{14}$, $D_8 \times C_2$, $D_{22}$, $D_8 \ast Q_8$ or $(D_8 \ast Q_8) \times C_2$.
\end{conj}

\end{document}